\documentclass[12pt]{amsart}

\usepackage{amssymb}
\usepackage{amsthm}
\usepackage{amsmath}
\usepackage{url}
\usepackage{xypic}
\usepackage{xy}
\usepackage{color}
\usepackage{marginnote}

\newtheorem{theorem}{Theorem}

\newtheorem{lemma}[theorem]{Lemma}
\newtheorem{corollary}[theorem]{Corollary}
\newtheorem{proposition}[theorem]{Proposition}

\newtheorem*{theorem*}{Theorem}

\DeclareMathOperator{\im}{im}

\newcommand{\Fq}{\mathbb{F}_q}
\newcommand{\F}{\mathbb{F}}
\newcommand{\G}{\mathbb{G}}
\renewcommand{\P}{\mathbb{P}}
\newcommand{\Q}{\mathbb{Q}}
\newcommand{\Z}{\mathbb{Z}}

\newcommand{\T}{\mathcal{T}}
\newcommand{\Fr}{\mathrm{Fr}} % Frobenius

\begin{document}

\title[Character sums for isogenies of low degree]{Character sums
  determined by low degree isogenies of elliptic curves} 

\author{Dustin Moody} %\corref{cor1}}
% \ead{dbmoody25@gmail.com}
% \address{National Institute of Standards and Technology (NIST),\\
%  100 Bureau Drive - Stop 8930, Gaithersburg, MD, 20899-8930, USA}
% \cortext[cor1]{corresponding author}

\author{Christopher Rasmussen}
% \ead{crasmussen@wesleyan.edu}
% \address[CR]{Department of Mathematics \& Computer Science, \\
%   Wesleyan University, Middletown, CT 06459, USA}

\begin{abstract}
   We consider character sums determined by isogenies of elliptic
   curves over finite fields. We prove a congruence condition for
   character sums attached to arbitrary cyclic isogenies, and produce
   explicit formulas for isogenies of small degree.
\end{abstract}

\maketitle

\section{Introduction}

Let $p > 3$ be prime, and let $h_p$ denote the class number of
$\Q(\sqrt{-p})$. For convenience, let $h_p^*$ denote either $0$ or
$h_p$, as $p$ is congruent to $1$ or $3$, respectively. A well-known
consequence of Dirichlet's class number formula is the relation
\begin{equation}
\label{eq:Dirichlet}
\sum_{x=1}^{(p-1)} x \left( \frac{x}{p} \right) = -ph_p^*.
\end{equation}
In the recent article \cite{MR}, McLeman and Rasmussen reinterpreted
this sum in the following way. Let $\G_m$ denote the multiplicative
group, and let $\varphi_2 \colon \G_m \to \G_m$ be
the squaring homomorphism. Then the Legendre symbol
$(\frac{\cdot}{p})$ is simply the cokernel character of the group
homomorphism $\varphi_2 \bigr|_{\G_m(\F_p)} \colon \G_m(\F_p) \to
\G_m(\F_p)$. Now, the left hand side of \eqref{eq:Dirichlet} may be
rewritten $\sum_{x \in \G_m(\F_p)} \{x\} (\frac{x}{p})$,
where $\{\cdot\}$ denotes an appropriately chosen lift $\F_p \to
\Z$. So this character sum attached to the homomorphism $\varphi_2$ is
divisible by $p$, and the quotient by $p$ computes $h_p$ when $p
\equiv 3 \pmod{4}$.  

McLeman and Rasmussen then applied this view to isogenies of elliptic
curves. Let $a, b \in \Z$ be chosen such that the following equations
give nonsingular Weierstrass models for elliptic curves over $\Q$: 
\begin{equation*}
\begin{split}
E_1 \colon y^2 & = x^3 + ax^2 + bx, \\
E_2 \colon y^2 & = x^3 - 2ax^2 + (a^2 - 4b)x.
\end{split}
\end{equation*}
Let $\tau \colon E_1 \to E_2$ denote the $2$-isogeny $(x,y) \mapsto
(\frac{y^2}{x^2}, \frac{y(b-x^2)}{x^2})$. For any prime $p$ of good
reduction, this descends to an isogeny on the reduced curves,
$\tilde{\tau}$, and further to a homomorphism of abelian groups,
$\tilde{\tau} \colon \tilde{E}_1(\F_p) \to \tilde{E}_2(\F_p)$. Let
$\chi_\tau \colon \tilde{E}_2(\F_p) \to \{\pm 1\}$ denote the
associated cokernel character, and define
\[ S_{\tau,p} := \sum_{\stackrel{P \in \tilde{E}_2(\F_p)}{P \neq \infty}}
\{x(P) - a \} \chi_\tau(P). \]
\begin{theorem*}[McLeman-Rasmussen, {\cite[Thm.~9]{MR}}]
For any prime $p > 3$ of good reduction, $S_{\tau,p}$ is divisible by
$p$. Moreover, $S_{\tau,p}$ approximates $-ph_p^*$ in the following
uniform sense: the difference $|-\frac{1}{p}S_{\tau,p} - h_p^*|$ is
bounded by a constant, independent of $p$.
\end{theorem*}
In fact, for most isogenies, the difference is zero for a positive
density of primes; see \cite{MR} for details. 

\subsection{Isogenies of higher degree}

It is natural to consider other isogenies over finite fields to see if
the analogous character sums also carry arithmetic
information. Already some progress has been made in this direction;
McLeman and Moody \cite{MM} have demonstrated a similar phenomenon
occurring in a family of $3$-isogenies of elliptic curves with complex
multiplication.

In this paper, we consider the case of cyclic
$m$-isogenies of elliptic curves. Here, there is some divergence
between the case of elliptic curves and the case of an endomorphism of
$\G_m$; this is in fact what makes the generalization
interesting. For example, take $m > 2$ odd and suppose $p \equiv 1 \pmod{m}$ is
prime. The $m$-th power map $\varphi_m \colon \G_m \to \G_m$ induces
an endomorphism on $\G_m(\F_p)$ whose cokernel character is the $m$-th
order residue symbol $(\frac{\cdot}{p})_m$. However, the analogous
character sum
\[ \sum_{x \in \G_m(\F_p)} \{x\} \left( \frac{x}{p} \right)_m \]
vanishes always (an immediate consequence of the observation that $-1$
is an $m$-th power in $\F_p^\times$).

However, in the case of an $m$-isogeny $\pi$ of elliptic curves
defined over $\F_p$ ($m>2$), the situation is more
complicated. Experimentally, one observes that the 
sum does not collapse to $0$ (even under the additional condition $p
\equiv 1 \pmod{m}$). One suspects these sums also contain
arithmetic information, and this idea is already corroborated by the
recent work of McLeman-Moody for $m=3$. In the current article, we
present new results on the behavior of these character sums attached
to isogenies of degree greater than two.
\begin{theorem}
Let $\pi \colon E_1 \to E_2$ be a V\'{e}lu isogeny (as defined below) of Weierstrass
elliptic curves defined over $\F_q$ of degree $m$. Let $\zeta \in
\overline{\F_p}$ be a primitive $m$-th root of unity. Suppose $P \in
E_1(\F_p)$ generates the kernel of $\pi$, and let $\chi$ denote the
cokernel character of the restriction $\pi \bigr|_{E_1(\F_p)} \colon
E_1(\F_p) \to E_2(\F_p)$. Then the weighted character sum $\sum_{Q \in
  E_2(\F_p)} x(Q) \chi(Q)$ over points of $E_2$ coincides with the sum 
$\sum_{j=1}^{m-1} \zeta^j x(jP)$, taken over certain points of
$E_1$. 
\end{theorem}
In fact, any non-constant separable isogeny may be realized as a
V\'{e}lu isogeny. Note that the number of terms in the
second sum does not grow with $p$.

As an application, we take standard families of isogenies of small
degree over finite fields, and compute the associated character sum explicitly. Sadly, we
cannot provide a characteristic zero result in the 
spirit of \cite{MR} and \cite{MM}; in those papers, the cokernel
character could be written explicitly in terms of the Weierstrass
models of the elliptic curves, by exploiting the properties of the
Tate pairing. There is some hope that a combination of these methods,
and those of the present article, may yet yield a similar result for
larger values of $m$. In contrast to \cite{MR} and \cite{MM}, however,
the results hold over arbitrary finite fields (not just prime fields),
and do not require the condition $p \equiv 1 \pmod{m}$.

This paper is organized as follows.  In \S2, we deduce an exact
sequence related to isogenies of elliptic curves over finite
fields. In \S3 we generalize the character sum from \cite{MR} to
higher degree isogenies and prove the main theorem, a relation that
transports character sums across a normalized isogeny. In \S4, we
apply the result to several well-known families of isogenies of small degree
and obtain congruence relations for the character sums in question.

\subsection{Notation}
We set some notation and recall some facts about elliptic curves. 
Let $k$ be a field. An elliptic curve over $k$ is a nonsingular genus
$1$ curve over $k$ with a $k$-rational base point. For any $k'/k$, the
set $E(k')$ is naturally an abelian group with the base point as
identity. Any elliptic curve over $k$ is isomorphic to a nonsingular plane
cubic in $\P^2$ which possesses an affine model in the form of a
Weierstrass equation: 
\begin{equation}\label{eq:WE}
y^2 +a_1 xy + a_3 y = x^3 + a_2 x^2 + a_4 x + a_6, \qquad a_i \in k.
\end{equation}
By a Weierstrass elliptic curve, we mean a nonsingular
projective cubic curve $E \subset \P^2$ of the form \eqref{eq:WE}. We
let $\infty$ denote the base point. We will at times refer to Weierstrass
elliptic curves $E'$, $E''$, \ldots of the same form, replacing the
coefficients $a_i$ of \eqref{eq:WE} with {$a_i'$,~$a_i''$, \ldots} respectively.

On any elliptic curve, the set of holomorphic and non-vanishing
differentials (together with $0$) is a $1$-dimensional
$\overline{k}$-vector space. For a Weierstrass elliptic curve $E$,
there is a distinguished generator for this space, the invariant
differential 
\begin{equation}
\omega = \frac{dx}{2y + a_1x + a_3}.
\end{equation}
Throughout the paper, we denote the invariant differentials of
Weierstrass elliptic curves $E$, $E'$, $\ldots$, by $\omega$,
$\omega'$, $\ldots$, respectively.

A curve morphism $\varphi \colon E \to E'$ which sends the base point
of $E$ to the base point of $E'$ is called an isogeny; it is always a
group homomorphism. If an isogeny $\varphi$ is defined over $k'$, then it
restricts to a homomorphism $E(k') \to E'(k')$ of abelian groups,
denoted $\varphi_{k'}$.  

% If $\varphi$ is a non-constant isogeny of degree $m$, there always
% exists an isogeny $\hat{\varphi} \colon E' \to E$ for which
% $\hat{\varphi} \circ \varphi = [m]$, the multiplication-by-$m$ map on
% $E$. We call $\hat{\varphi}$ the dual isogeny to $\varphi$; it also
% satisfies $\varphi \circ \hat{\varphi} = [m]$ on $E'$.

If $\varphi$ is an isogeny of two Weierstrass elliptic
curves, and $\omega'$ is the invariant differential of $E'$, then
$\varphi^* \omega'$ is again a holomorphic and non-vanishing
differential on $E$. Hence there exists $c_\varphi \in \overline{k}$
such that $\varphi^* \omega' = c_\varphi \omega$. If $c_\varphi = 1$,
we say that $\varphi$ is a normalized isogeny.

\section{A Useful Exact Sequence}

Let $E/k$ be an elliptic curve. For any finite subgroup $\mathcal{F}
\leq E(\overline{k})$, there exists an elliptic curve $E'$ and a
separable isogeny $\pi \colon E \to E'$ whose kernel is precisely
$\mathcal{F}$ (\cite[III.4.12]{Sil}). The curve $E'$ is unique up to
isomorphism.
\begin{proposition}\label{prop:ex_seq}
Let $k$ be a finite field, $E/k$ an elliptic curve, and let $\pi
\colon E \to E'$ be a separable isogeny with kernel $\mathcal{F} \leq
E(k)$. Then there exists a $k$-rational isogeny $\pi' \colon E' \to E$
such that the sequence 
\[ \xymatrix{
0 \ar[r] & \mathcal{F} \ar[r] & E(k) \ar[r]^{\pi_k} & E'(k) \ar[r]^{\pi'_k} &
\mathcal{F} \ar[r] & 0} \]
is exact. 
\end{proposition}
\begin{proof}
  Let $\Fr_E$ and $\Fr_{E'}$ denote the Frobenius endomorphisms $(x,y)
  \mapsto (x^q, y^q)$ on $E$ and $E'$, respectively. Since $\ker \pi =
  \mathcal{F} \leq E(k) = \ker (1 - \Fr_E)$, there exists, up to
  isomorphism, a unique isogeny $\pi' \colon E' \to E$ such that $1 -
  \Fr_E = \pi' \circ \pi$ (\cite[III.4.11]{Sil}); $\pi'$ is
  necessarily defined over $k$ since both $\pi$ and $1 - \Fr_E$
  are. Set $\mathcal{G} = \im \pi'_k$.

We have $\pi'_k \circ \pi_k = 0$; so $\im \pi_k \leq \ker \pi'_k$. Let $P' \in \ker
\pi'_k$. As $\pi$ is surjective, there exists $P \in E(\overline{k})$
such that $\pi(P) = P'$. Then $P \in \ker (1-\Fr_E) = E(k)$, and
$P' = \pi_k(P) \in \im \pi_k$. Thus, $\ker \pi'_k = \im \pi_k$, and it
remains only to show $\mathcal{F} = \mathcal{G}$.

Since $\pi$ is defined over $k$, it commutes with the Frobenius
endomorphisms (i.e., $\pi \circ \Fr_E = \Fr_{E'} \circ \pi$). Thus, we
have
\begin{equation*}
(\pi \circ \pi') \circ \pi = \pi \circ (\pi' \circ \pi) = \pi \circ
(1 - \Fr_E) = (1 - \Fr_{E'}) \circ \pi. 
\end{equation*}
As $\pi$ is surjective, $\pi \circ \pi' = 1 -
\Fr_{E'}$. Thus, $\pi_k \circ \pi'_k = 0$, and so $\mathcal{G} = \im
\pi'_k \leq \ker \pi_k = \mathcal{F}$. On the other hand, 
$k$-isogenous elliptic curves have the same number of points over $k$
(\cite[Thm.~1(c)]{Tate}) and $\im \pi_k = \ker \pi'_k$. Thus, 
\begin{equation}
\#\mathcal{G} = \frac{ \# E'(k) }{\# \ker \pi'_k} = \frac{ \#E(k)
}{\#\im \pi_k} = \#\mathcal{F},
\end{equation}
so $\mathcal{G} = \mathcal{F}$ and the sequence is exact, as claimed.
\end{proof}

\subsection{V\'elu's formula}

Let $E$ be a Weierstrass elliptic curve defined over an arbitrary
field $k$, and let $\mathcal{F} \leq E(\overline{k})$ be a finite
subgroup. Let $\mathcal{F}^* = \mathcal{F} -
\{\infty\}$. Then there exists an elliptic curve $E'$ and a separable
isogeny $E \to E'$ whose kernel is $\mathcal{F}$. For any choice of
$E$ and $\mathcal{F}$, V\'{e}lu has given explicit formulas for both a
Weierstrass elliptic curve $E'$ and an isogeny $V \colon E \to E'$
with kernel $\mathcal{F}$. Let us write $P = (x_P, y_P)$ for the
affine coordinates of any point on $E$. If $P \notin \mathcal{F}$, the
coordinates of $V$ are given by 
\begin{equation}\label{eq:Velu}
\begin{split}
x_{V(P)} & = x_P + \sum_{Q \in \mathcal{F}^*} (x_{P+Q} - x_Q), \\
y_{V(P)} & = y_P + \sum_{Q \in \mathcal{F}^*} (y_{P+Q} - y_Q).
\end{split}
\end{equation}
If $E$ has the form \eqref{eq:WE}, then the Weierstrass equation for
$E'$ is
\begin{equation*}
E' \colon y^2 + a_1'xy + a_3'y = x^3 + a_2'x^2 + a_4'x + a_6',
\end{equation*}
where $a_i' = a_i$ for $1 \leq 3$, and $a_4', a_6' \in \overline{k}$
are determined explicitly by the points of $\mathcal{F}$. The formulas
in \cite{Vel} make it clear -- as expected~-- that $a_4', a_6'
\in k$ if $\mathcal{F}$ is a $k$-rational subgroup of
$E(\overline{k})$ (i.e., invariant under the action of
$\mathrm{Gal}(\overline{k}/k)$). Moreover, it follows directly from
the formulas (for example, see \cite{Shu}) that $V$ is always a normalized
isogeny. Throughout this article, we will refer to an isogeny of the
form $V$ as a V\'{e}lu isogeny.

Let $\varphi \colon C \to C'$ be a separable isogeny of elliptic
curves, and fix an isomorphism $\eta \colon C \to E$ identifying $C$
with a Weierstrass elliptic curve $E$. Let $V \colon E \to E'$
denote the V\'{e}lu isogeny with kernel $\eta(\ker \varphi)$. Then
there exists (\cite[III.4.11]{Sil}) an isomorphism $\eta' \colon C' \to E'$ such
that $\eta' \varphi = V \eta$. In this way, any separable isogeny may
be realized as a V\'{e}lu isogeny of Weierstrass elliptic curves,
a fact we will use in the sequel.

\subsection{Normalization}

Consider again the setting of Proposition \ref{prop:ex_seq}. The
purpose of the current section is to verify the following fact: if $E$
and $E'$ are Weierstrass elliptic curves and $\pi$ is a V\'elu
isogeny, then $\pi'$ is also a V\'elu isogeny.

\begin{lemma}
Suppose $\varphi \colon E \to E'$ is a separable isogeny of
Weierstrass elliptic curves defined over $k$, and let $V \colon E \to
E''$ denote V\'{e}lu's isogeny which has the same domain and kernel as
$\varphi$. Then $\varphi = V$ if and only if $\varphi$ is normalized
and $a_i' = a_i$ for $1 \leq i \leq 3$. 
\end{lemma}
\begin{proof}
One direction is immediate from the properties of $V$. For the other,
suppose $\varphi^* \omega' = \omega$ and $a_i' = a_i$ for $1 \leq i
\leq 3$. By \cite[III.4.11]{Sil}, there exists an isogeny $\eta$
such that $\varphi = \eta \circ V$. Comparing degrees, $\eta$ is
an isomorphism of Weierstrass elliptic curves, and so has the form
\[ \eta(x,y) = \bigl( u^{-2}(x-r), u^{-3}(y - sx + rs - t) \bigr),
\qquad r,s,t \in k, \quad u \in k^\times. \]
As $V$ and $\varphi$ are both normalized, $\eta$ is also. But
$\eta^*\omega'' = u\omega'$, so $u = 1$. The well-known formulas
relating $a_i'$ and $a_i''$ (\cite[pg.~45]{Sil}), together with the
equalities $a_i' = a_i = a_i''$ for $1 \leq i \leq 3$, now force $r =
s = t = 0$. Thus $\eta$ is the identity, $E' = E''$, and $V =
\varphi$. 
\end{proof}
\begin{corollary}\label{cor:ex_seq}
Take the hypotheses of Proposition \ref{prop:ex_seq}. If $E$, $E'$ are
Weierstrass elliptic curves and $\pi$ is a V\'{e}lu isogeny, then
$\pi'$ is also a V\'{e}lu isogeny.
\end{corollary}
\begin{proof}
As the Frobenius map is inseparable, we have $\Fr_E^*\omega =
0$; it follows that $\pi' \circ \pi = 1 - \Fr_E$ is normalized. As $\pi$ is
normalized by assumption, $\pi'$ must be normalized also. But we already
have $a_i = a_i'$ for $1 \leq i \leq 3$, and so we are done by the
previous lemma.
\end{proof}

\section{A character sum determined by an isogeny}

We now consider a higher degree analogue of the character sums in
\cite{MR}, where the isogenies had degree $m=2$. Fix a prime $p$ and
set $q = p^r$ for some $r > 0$. Let $m \geq 2$ satisfy $(p,m) =
1$. (Note, however, that we do not require $p \equiv 1 \pmod{m}$ as in
\cite{MM} or even \cite[\S2]{MR}.) Suppose $E_1/\F_q$ is a Weierstrass
elliptic curve and $P \in E_1(\F_q)$ has exact order $m$. Let $V
\colon E_1 \to E_2$ denote the V\'{e}lu isogeny with kernel $\langle P
\rangle$. By Corollary \ref{cor:ex_seq},
\[ \xymatrix{ 0 \ar[r] & \langle P \rangle \ar[r] & E_1(\F_q)
  \ar[r]^{V^{\phantom{1}}_{\F_q}} & E_2(\F_q) \ar[r]^{V'_{\F_q}} & \langle P \rangle
  \ar[r] & 0} \]
is exact, where $V' \colon E_2 \to E_1$ is the V\'{e}lu isogeny with
kernel $V(E_1(\F_q))$. Thus, $V'_{\F_q}$ induces an isomorphism
$E_2(\F_q)/V(E_1(\F_q)) \stackrel{\cong}{\longrightarrow} \langle P
\rangle$. Fix $Q \in E_2(\F_q)$ such that $V'(Q) = P$. (Of course, if $m$ is
prime and we allow ourselves to replace $P$ with another generator of
$\langle P \rangle$, then we may choose any ${Q \in E_2(\F_q) -
  V(E_1(\F_q))}$.) 

Let $\zeta$ denote a primitive $m$-th root of unity in
$\overline{\F_q}$. The cokernel character for $V_{\F_q}$ may be
realized explicitly as $\chi_{\zeta,P} \colon E_2(\F_q) \to
\mu_m(\overline{\F_q})$, defined for any $R \in E_2(\F_q)$ by
\[ \chi_{\zeta,P}(R) := \zeta^j, \quad \mbox{$j$ chosen such that $R - jQ \in
  V(E_1(\F_q))$.}\]
The definition does not depend on the choice of $Q$; if $V'(Q_1) =
V'(Q) = P$, then $R - jQ_1$ and $R - jQ$ differ by an element of $\ker
V'_{\F_q} = \im V_{\F_q}$. The definition does depend on both
the choice of generator $P$ for the kernel of $V$ and also the choice
of generator $\zeta$ of $\mu_m$, but these choices are related as
follows. If $\langle P' \rangle = \langle P \rangle$, then there
exists $a$, $(a,m) = 1$, such that $P' = aP$. Further, there exists
$\zeta' \in \mu_m$ such that $\zeta'^a = \zeta$. The point $Q' = aQ$
satisfies $V'(Q') = P'$, and for any $R$ we have $R - jQ' = R -
jaQ$. Thus, if $\chi_{\zeta, P'}(R) = \zeta^j$, then $R - jaQ \in
V(E_1(\F_q))$, and $\chi_{\zeta', P}(R) = \zeta'^{aj} = \chi_{\zeta,
  P'}(R)$. 

So we fix, once and for all, a generator $\zeta$ for $\mu_m$, and
write $\chi_P$ for $\chi_{\zeta,P}$. We define an associated character
sum over the $\F_q$-points of $E_2$:
\[ S_P := \sum_{\substack{R \in E_2(\Fq) \\ R \neq \infty}} \chi_P(R)
x_R. \]
The main result of this paper is that $S_P$ is determined by a
simpler weighted sum over the $x$-coordinates of points in $\langle P
\rangle$.  
\begin{theorem}\label{thm:main}
With the above notation, $\displaystyle S_P = \sum_{j=1}^{m-1} \zeta^j x_{jP}$.
\end{theorem}
Thus, the task of computing the sum $S_P$, a priori involving roughly
$q \approx \#E_2(\F_q)$ terms, in fact only involves summing $m-1$
terms, and one does not even need to know $E_2$ explicitly in order to
do it!

\begin{proof}[Proof of Theorem \ref{thm:main}]
Let $\mathcal{T} = V(E_1(\F_p))$, and set $\mathcal{T}^* =
\mathcal{T} - \{\infty \}$. If $m = 2k$ is even, then
\[ S_P=\sum_{R \in \mathcal{T}^*} x_R + \sum_{j=1}^{k-1}
\left(\sum_{R \in {\T}}(\zeta^j+\zeta^{-j})x_{jQ+R}\right)+\sum_{R \in
  {\T}} \zeta^k x_{kQ+R}. \]
The kernel of $V'$ is $\mathcal{T}$. So from V\'{e}lu's formula
\eqref{eq:Velu} (and noting $\zeta^k = -1$), we have
\begin{equation*}
\begin{split}
S_P & = -x_{kQ} + \sum_{R \in \mathcal{T}^*} (x_R-x_{kQ+R}) +
\sum_{j=1}^{k-1} \left(\sum_{R \in \T}(\zeta^j + \zeta^{-j}) x_{jQ+R}
\right) \\ 
 & = -x_{kV'(Q)} + \sum_{j=1}^{k-1} \left(\sum_{R \in
     \T} (\zeta^j + \zeta^{-j}) x_{jQ+R} \right) \\
\end{split}
\end{equation*}
\begin{equation*}
\begin{split}
\phantom{S_P} & = -x_{kP} + \sum_{j=1}^{k-1} (\zeta^j+\zeta^{-j}) \sum_{R \in
   \T} x_{jQ+R} \\
& = -x_{kP} + \sum_{j=1}^{k-1} (\zeta^j + \zeta^{-j}) \left(
   \sum_{R \in \mathcal{T}} x_{jQ+R} - \sum_{R \in \mathcal{T}^*}
   x_R+\sum_{R \in \mathcal{T}^*} x_R \right) \\ 
& = -x_{kP} + \sum_{j=1}^{k-1} (\zeta^j + \zeta^{-j}) \left(
   x_{V'(jQ)} + \sum_{R \in \T^*} x_R \right) \\
& = -x_{kP}+\sum_{j=1}^{k-1} (\zeta^j+\zeta^{-j})
x_{jP}+\sum_{j=1}^{k-1} (\zeta^j+\zeta^{-j})\sum_{R \in \mathcal{T}^*} x_R \\
& = -x_{kP} + \sum_{j=1}^{k-1} (\zeta^j+\zeta^{-j}) x_{jP} -(\zeta^{k}
+ 1) \sum_{R \in \mathcal{T}^*}x_R \\ 
& = -x_{kP}+\sum_{j=1}^{k-1} (\zeta^j+\zeta^{-j}) x_{jP} =
\sum_{j=1}^{m-1}\zeta^j x_{jP}. 
\end{split}
\end{equation*}
If $m = 2k+1$ is odd, the argument is similar. We have
\begin{equation*}
\begin{aligned}
S_P & = \sum_{R \in \mathcal{T}^*} x_R + \sum_{j=1}^{k} \left(\sum_{R
    \in {\T}} (\zeta^j + \zeta^{-j}) x_{jQ+R} \right) \\
& = \sum_{R \in \mathcal{T}^*} x_R + \sum_{j=1}^k (\zeta^j +
\zeta^{-j}) \left( \sum_{R \in \mathcal{T}} x_{jQ+R} - \sum_{R \in
    \mathcal{T}^*} x_R + \sum_{R \in \mathcal{T}^*} x_R \right) \\
& = \sum_{R \in \mathcal{T}^*} x_R + \sum_{j=1}^k (\zeta^j+\zeta^{-j})
\left(x_{V'(jQ)} + \sum_{R \in \mathcal{T}^*} x_R \right) \\
& = \left(1 + \sum_{j=1}^k(\zeta^j+\zeta^{-j}) \right) \sum_{R \in
  \mathcal{T}^*} x_R + \sum_{j=1}^k (\zeta^j+\zeta^{-j})x_{jP} \\
& = \sum_{j=1}^k (\zeta^j+\zeta^{-j})x_{jP} = \sum_{j=1}^{m-1} \zeta^j x_{jP},
\end{aligned}
\end{equation*}
as claimed.
\end{proof}

\section{Applications}

\subsection{General formulas for small $m$}

As discussed in the previous section, the value of $S_P$ is dependent
on the choice of generator $\zeta$ for $\mu_m$. Hence, there may be up to
$\phi(m)$ possible values for the sum. In fact, as $x_{-P} =
x_P$ for every point $P$ on a Weierstrass elliptic curve, there are
only $\phi(m)/2$ distinct values for $S_P$ for $m > 2$.

Let $\lambda_i = \zeta^i + \zeta^{-i}$; then we may rewrite Theorem
\ref{thm:main} as
\[ S_P = \begin{cases} \phantom{-x_{kP} + }\sum_{j=1}^k \lambda_j x_{jP} & m = 2k+1 \\
  -x_{kP} + \sum_{j=1}^{k-1} \lambda_j x_{jP} & m = 2k \end{cases}. \]
For any particular value of $m$, the expression may simplify further
due to the symmetries of $\mu_m$. For example, when $m = 8$,
$\lambda_2 = 0$ and $\lambda_3 = -\lambda_1$. In Table
\ref{table:SP}, we catalog these expressions for $S_P$ for various
small $m$.
\begin{table}[t]
\caption{$S_P$ for various $m$}
\label{table:SP}
\begin{tabular}{crcc}
\hline \hline
$\phi(m)$ & $m$ & & $S_P$ \\
\hline
$1$ & $2$ & & $-x_P$ \\
\hline
$2$ & $3$ & & $-x_P$ \\
    & $4$ & & $-x_{2P}$ \\
    & $6$ & & $-x_{3P}$ \\
\hline
$4$ & $5$ & & $\lambda_1 x_P + \lambda_2 x_{2P}$ \\
    & $8$ & & $\lambda_1 (x_P - x_{3P}) - x_{4P}$ \\
    & $10$ & & $\lambda_1(x_P - x_{4P}) + \lambda_2(x_{2P} - x_{3P}) -
    x_{5P}$ \\
    & $12$ & & $\lambda_1(x_P - x_{5P}) + x_{2P} - x_{4P} - x_{6P}$ \\
\hline
$6$ & $7$ & & $\lambda_1 x_P + \lambda_2 x_{2P} + \lambda_3 x_{3P}$ \\
    & $9$ & & $\lambda_1 x_P + \lambda_2 x_{2P} - x_{3P} + \lambda_4
    x_{4P}$ \\
\hline \hline
\end{tabular}
\end{table}

\subsection{Parametrized families of isogenies}
For the values of $m$ given in Table \ref{table:SP}, there are
explicit parametrizations of Weierstrass elliptic curves with a
rational point of order $m$; thus, each such curve is the domain of a
V\'{e}lu isogeny of degree $m$ which satisfies the hypotheses of
Proposition \ref{prop:ex_seq}. Consequently, we may compute $P$, and
hence $S_P$, directly from the coefficients of the Weierstrass
equation. 

As a final application, we find these values of $S_P$ for parametrized
families for $m \in \{2, 3, 4, 5, 6, 8 \}$. The results are given in Table
\ref{table:families}. This corresponds to the computation of
$S_{\tau,p}$ in \cite{MR} for families of degree $2$ isogenies,
although our sums here are only valid in positive characteristic. Even
in the case $m=2$, however, the present result generalizes \cite{MR},
where the characteristic $p$ formulas were only established over the
prime field $\F_p$.

Most parametrizations are given in \cite[Table 3]{Kub}; in other
cases, a reference is listed in the table. As in \cite{Kub}, we let
$E(\beta, \gamma)$ denote the Weierstrass elliptic curve
\[ E(\beta, \gamma) \colon y^2 + (1 - \gamma)xy - \beta y = x^3 -
\beta x^2. \]
The parameters $\beta$ and $\gamma$ are constrained only by the
condition that the resulting equation must be nonsingular. For each
family, we identify a point $P$ of order $m$ and the value(s) of $S_P$
for the V\'{e}lu isogeny of kernel $\langle P \rangle$.
\begin{table}[t]
\caption{$S_P$ for parametrized families.}
\label{table:families}
\begin{tabular}{lcc}
\hline \hline
$m = 2$ & & $E \colon y^2 = x(x^2 + \alpha x + \beta)$; $P = (0,0)$ \\
 & & $S_P = 0$ \\
\hline \hline
$m = 3$ & & $E \colon y^2 + \alpha xy + \beta y = x^3$; $P = (0,0)$ \\
 & & $S_P = 0$ \\
\hline
$m = 4$ & & $E \colon y^2 + xy - \alpha y = x^3 - \alpha x^2$; $P = (0,0)$ \\
 & & $S_P = -\alpha$ \\
\hline
$m = 4$ & & $E \colon y^2 = x^3 + (1 - 2\alpha)x^2 + \alpha^2 x$; $P =
(\alpha, \alpha)$ \\
\cite[Lem.~3.1]{FMW} & & $S_P = 0$ \\
\hline
$m = 6$ & & $E = E(\alpha + \alpha^2, \alpha)$; $P = (0,0)$ \\
 & & $S_P = -\alpha(\alpha + 2)$ \\
\hline \hline
$m = 5$ & & $E = E(\alpha, \alpha)$; $P = (0,0)$ \\
 & & $S_P = \lambda_2 \alpha$ \\
\hline
$m =8$ & & $E = E(\beta, \gamma)$; $P = (0,0)$ \\
 & &  $\beta = (2\alpha - 1)(\alpha - 1)$,
$\gamma = \beta/\alpha$ \\
& & $S_P = -\alpha^{-1}(\alpha - 1)(\alpha^2 + 2\lambda_1 \alpha -
\lambda_1)$ \\
\hline \hline
\end{tabular}
\end{table}
Of course, this computation can be done for any parametrized family,
although the corresponding formulas are no longer particularly
informative as $m$ gets large. For example, when $m = 10$, Kubert
provides the parametrization (\cite[Table 3]{Kub}): 
\[E = E(\beta, \gamma), \quad \delta = \alpha(\alpha - 1) - 1, \gamma
= \alpha(\delta - 1), \beta = \gamma \delta, \] 
which possesses $P = (0,0)$ as point of order $10$. For this curve, we
find 
\begin{equation*}
\begin{split}
S_P & = \lambda_2 \alpha^5  - (3\lambda_2 + 2)\alpha^4 - (\lambda_2  - 4)\alpha^3\\
& \qquad + (6\lambda_2 + 2)\alpha^2 + (\lambda_2 - 4)\alpha - (2\lambda_2 + 2).
\end{split}
\end{equation*}

\section{Future Work}
As previously noted, when $\phi(m) = 2$ the value of $S_P$ is
well-defined and we have explicit expressions for it.  These cases
($m=3,4,6$) are naturally the easiest to research if relations can be
found between the quotient $S_p/p$ and the class number $h_p$. The
authors are currently investigating the cases $m=4$ and $m=6$. In
another direction, we ask whether any of the results generalize to
character sums attached to isogenies of higher dimensional abelian
varieties.

\subsection*{Acknowledgments}
We acknowledge the contribution of SAGE \cite{SAGE}, which facilitated
the construction of examples which were helpful in discovering the
main theorems of this work. The first author would also like to thank
Daniel Shumow for conversations which led to a better understanding of
normalized isogenies.  Finally, we greatly appreciate the comments of
an anonymous reviewer for many useful suggestions, including comments
that greatly improved the presentation in \S2. This research
was partially supported by the National Institute of Standards and
Technology.

\bibliography{MoodyRasmussen}{}
\bibliographystyle{abbrv}

\end{document}